\documentclass[a4paper,12pt]{amsart}
\usepackage{mathtools}
\usepackage{amsmath, amssymb, amscd, amsthm, amsfonts, amsbsy}
\usepackage{array}
\usepackage{epsfig}
\usepackage{ragged2e}
\usepackage{graphics}
\usepackage{fancyhdr}
\usepackage{xcolor}
\usepackage[all,cmtip]{xy} 
\usepackage{ifthen}
\usepackage[pdftex,bookmarks=true]{hyperref}
\nonstopmode \numberwithin{equation}{section}
\makeatletter
\newcommand{\ubrace}[2]{\mathord{\mathpalette\ubrace@{{#1}{#2}}}}
\newcommand{\ubrace@}[2]{\ubrace@@#1#2}
\newcommand{\ubrace@@}[3]{
	\underbrace{#1#2}_{#3}%
}
\makeatother

\theoremstyle{definition}

\newtheorem{Thm}{Theorem}[section]
\newtheorem{Cor}[Thm]{Corollary}
\newtheorem{Lem}[Thm]{Lemma}
\newtheorem{Prop}[Thm]{Proposition}

\theoremstyle{definition}

\newtheorem{Rem}[Thm]{Remark}
\newtheorem{Ex}[Thm]{Example}
\begin{document}
	\bibliographystyle{amsplain}
	\bibliographystyle{amsplain}
	\title{A Structure Theorem for Centralizers of Dilations in $QI(\mathbb{R}_{+})$}
	\author{Swarup Bhowmik and Deblina Das}
	\address{Swarup Bhowmik, Department of Mathematics,
		Indian Institute of Science Education and Research Bhopal, India.}
	\email{swarupbhowmik@iiserb.ac.in}
	\address{Deblina Das, Department of Mathematics, Indian Institute of Technology Palakkad}
	\email{212114002@smail.iitpkd.ac.in, deblina099@gmail.com}

			\begin{abstract}
				We study centralizers of dilations in the quasi-isometry group of the positive real line. We introduce an asymptotic invariant defined via coarsely dense sequences at infinity and establish a rigidity theorem for quasi-isometries that coarsely commute with a dilation. As an application, we identify the subgroup of the centralizer consisting of elements with non-empty asymptotic invariant and prove that it is naturally isomorphic to the multiplicative group of positive real numbers.
			\end{abstract}
		
%


	\maketitle
	\pagestyle{myheadings}
	\markboth{Swarup Bhowmik and Deblina Das}{A Structure Theorem for Centralizers of Dilations in $QI(\mathbb{R}_{+})$}
	\bigskip

	\section{Introduction}
	
		The quasi-isometry group of a metric space encodes its large-scale geometric
	symmetries and plays an important role in geometric group theory. Explicit investigations of $QI(X)$ are known only for a limited class of metric spaces $X$,
	including solvable Baumslag-Solitar groups $BS(1,n)$ \cite[Theorem 7.1]{Farb Mosher}, the groups $BS(m,n),~1<m<n$ \cite[Theorem 4.3]{Whyte}, irreducible lattices in semisimple Lie groups (see \cite{Farb} and the references therein) etc. Even in the
	simplest cases, such as the real line $\mathbb R$ \cite{Ye Zhao} and the positive half-line
	$\mathbb R_+$ \cite{Bhowmik Chakraborty 2}, the quasi-isometry group is large and exhibits nontrivial
	algebraic behavior \cite{Sankaran}. Consequently, contemporary research has shifted toward uncovering the structure of its natural subgroups.

	 Despite the fact that the center of $QI(\mathbb{R})$ is trivial \cite{Chakraborty},
	the structure of its natural subgroups, in particular centralizers, remains
	largely unexplored. In this article, we study centralizers in $QI(\mathbb R_+)$ associated to dilation maps
	$$
	f_\lambda(x)=\lambda x,
	\qquad \lambda>1.
	$$
	More precisely, we investigate quasi-isometries that commute with a dilation up to bounded distance and examine how this coarse commutation property constrains their asymptotic behavior.
	
	A primary obstruction to analyzing the full centralizer $C([f_\lambda])$ is the presence of elements in $QI(\mathbb{R}_{+})$, the quasi isometry group of the positive real line with wildly oscillating local or global behavior. To extract meaningful asymptotic data, we introduce an invariant $I[g]$ (see \ref{Asymptotic Invariant}) associated to $[g]\in QI(\mathbb{R}_{+})$  defined via limits along coarsely dense sequences at infinity. 
	  We show that this invariant is well defined on quasi-isometry classes and establish a rigidity theorem for quasi-isometries that coarsely commute with a dilation. Specifically, we prove that coarse commutation with a dilation imposes a rigidity phenomenon, forcing all convergent asymptotic ratios arising from coarsely dense sequences to coincide. 
	  
	Motivated by this rigidity phenomenon, we consider the subgroup
	$$
	C^{*}([f_\lambda])
	=
	\left\{
	[g]\in C([f_\lambda]): I[g]\neq\varnothing
	\right\}
	$$
	of the centralizer consisting of elements with non-empty asymptotic invariant. Our main result provides a complete algebraic description of this subgroup.
	
	\begin{Thm}\label{main}
		Let $\lambda>1$. Then
		$$
		C^{*}([f_\lambda])
		\cong
		\mathbb R_{+}.
		$$
	\end{Thm}
	
	Thus, although the ambient group $QI(\mathbb R_+)$ is highly nontrivial, the subgroup of the centralizer detected by the asymptotic invariant admits a remarkably simple algebraic description.
	
	The paper is organized as follows. In Section 2, we recall the necessary background on quasi-isometries of $\mathbb R_+$ and coarsely dense sequences. Section 3 develops the asymptotic invariant and proves the rigidity theorem. In Section 4, we determine the structure of $C^{*}([f_\lambda])$ and establish the main theorem.

\section{Preliminaries}

\subsection{Quasi-isometries and the quasi-isometry group}

Let $(X,d)$ be a metric space. A map $f:X\to X$ is called a \emph{quasi-isometry}
if there exist constants $\mu\ge 1$, $\delta\ge 0$, and $M>0$ such that

\begin{enumerate}
	\item for all $x,y\in X$,
	\[
	\frac{1}{\mu}d(x,y)-\delta \le d(f(x),f(y)) \le \mu d(x,y)+\delta,
	\]
	\item every point of $X$ lies within distance $M$ of the image $f(X)$.
\end{enumerate}

Two quasi-isometries $f,g:X\to X$ are said to be \emph{equivalent}, written
$f\sim g$, if
\[
\sup_{x\in X} d\bigl(f(x),g(x)\bigr)<\infty.
\]
The set of equivalence classes of quasi-isometries under composition forms a
group, denoted by $QI(X)$, called the \emph{quasi-isometry group} of $X$.

Every quasi-isometry $f:X\to X$ admits a quasi-inverse $f^{-1}:X\to X$, unique up
to bounded distance, such that
\[
d\bigl(f^{-1}(f(x)),x\bigr)\le C
\quad\text{and}\quad
d\bigl(f(f^{-1}(x)),x\bigr)\le C
\]
for all $x\in X$ and some constant $C>0$.

\subsection{Quasi-isometries of $\mathbb R_+$}

Throughout this article, we work with the metric space
\[
\mathbb R_+=(0,\infty)
\]
equipped with the standard Euclidean metric. We denote by $QI(\mathbb R_+)$ the
quasi-isometry group of $\mathbb R_+$.\\
A basic family of quasi-isometries of $\mathbb R_+$ is given by the dilations
$
f_\lambda(x)=\lambda x,~ \lambda>0.
$\\
Different values of $\lambda$ define distinct elements of $QI(\mathbb R_+)$.

\subsection{Coarsely dense sequences at infinity}

A sequence $\{x_n\}\subset\mathbb R_+$ is said to be \emph{coarsely dense at
	infinity} if there exist constants $D>0$ and $T_0>0$ such that for every
$t\ge T_0$ there exists an index $n$ with
\[
|x_n-t|\le D.
\]
In particular, every coarsely dense sequence diverges to infinity, though the
converse need not hold.\\ Note that if $\{x_n\}\subset\mathbb R_+$ is coarsely dense at infinity and $g\in QI(\mathbb{R}_+)$, then the
sequence $\{g^{-1}(x_n)\}$ is also coarsely dense at infinity. This follows from
the coarse surjectivity and coarse Lipschitz properties of quasi-isometries.

\subsection{Asymptotic equivalence of sequences}

Two sequences $\{a_n\},\{b_n\}\subset\mathbb R_+$ are said to be
\emph{asymptotically equivalent} if
\[
\lim_{n\to\infty}\frac{a_n}{b_n}=1.
\]
This notion allows one to transfer asymptotic information between sequences
under quasi-isometries.

\subsection{Centralizers in $QI(\mathbb R_+)$}

For an element $[f]\in QI(\mathbb R_+)$, its \emph{centralizer} is defined by
\[
C([f])=\{[g]\in QI(\mathbb R_+): [g][f]=[f][g]\}.
\]
Centralizers provide an important information about the internal algebraic
structure of $QI(\mathbb R_+)$.
In this article, we are primarily interested in centralizers of dilation classes
$[f_\lambda]$ with $\lambda>1$.

	\section{An Asymptotic Invariant and a Rigidity Theorem}

In this section, we introduce an asymptotic invariant associated to quasi-isometries of $\mathbb R_+$ and establish a rigidity theorem for elements of the centralizer of a dilation.

\subsection{An Asymptotic Invariant}\label{Asymptotic Invariant}

For any $[g]\in QI(\mathbb R_+)$, define
$$
I[g]
=
\left\{
\lim_{n\to\infty}
\frac{x_n}{g^{-1}(x_n)}
:
\{x_n\}\text{ is coarsely dense at infinity and the limit exists}
\right\}.
$$
We first show that $I[g]$ depends only on the quasi-isometry class of $g$.

\begin{Prop}
	\label{prop:invariant-well-defined}
	The invariant $I{[g]}$ is well defined on $QI(\mathbb{R}_+)$.
\end{Prop}

\begin{proof}
	Suppose $[f]=[g]$. Then it can be easily checked that $f^{-1}\sim g^{-1}$ .\\
	Let $\lambda\in I{[g]}$. By definition, there exists a coarsely dense sequence
	$\{x_n\}$ with $x_n\to\infty$ such that
	\[
	\lim_{n\to\infty}\frac{g^{-1}(x_n)}{x_n}=\frac{1}{\lambda}.
	\]
	Since $|f^{-1}(x_n)-g^{-1}(x_n)|$ is uniformly bounded, we have
	\[
	\left|
	\frac{f^{-1}(x_n)}{x_n}-\frac{1}{\lambda}
	\right|
	\le
	\frac{|f^{-1}(x_n)-g^{-1}(x_n)|}{x_n}
	+
	\left|
	\frac{g^{-1}(x_n)}{x_n}-\frac{1}{\lambda}
	\right|
	\longrightarrow 0~\text{as~}n\to\infty.
	\]
	Thus $\lambda\in I{[f]}$. By symmetry, $I{[f]}\subset I{[g]}$, and hence
	$I{[f]}=I{[g]}$.
\end{proof}
\subsection{Rigidity for Centralizers of Dilations}

We now study the invariant $I[g]$ for elements of the centralizer of a dilation. The following lemmas culminate in a rigidity theorem establishing the uniqueness of the asymptotic invariant.

\begin{Lem}
	\label{lem:scale-matching}
	Let $g:\mathbb{R}_+\to\mathbb{R}_+$ be a quasi-isometry, let $\lambda>1$ be a fixed real number,
	and let $m\in\mathbb{Z}$. Let $\{y_n\}\subset\mathbb{R}_+$ be any sequence with $y_n\to\infty$,
	and let $\{x_n\}\subset\mathbb{R}_+$ be a sequence that is coarsely dense at infinity.
	Then there exists a subsequence $\{x_{r_n}\}$ of $\{x_n\}$ such that
	\[
	\lim_{n\to\infty}\frac{\lambda^m\, g^{-1}(x_{r_n})}{y_n}=1.
	\]
\end{Lem}

\begin{proof}
	Since $g\in QI(\mathbb{R}_+)$, its quasi-inverse $g^{-1}$ is also a quasi-isometry.
	Hence there exists a constant $C>0$ such that for all sufficiently large $t$,
	\[
	|g^{-1}(g(t)) - t| \le C
	\quad \text{and} \quad
	|g(g^{-1}(t)) - t| \le C.
	\]
	Moreover, there exist constants $K\ge1$ and $A>0$ such that
	\[
	|g^{-1}(u) - g^{-1}(v)| \le K|u-v| + A
	\quad \text{for all sufficiently large } u,v\in\mathbb{R}_+.
	\]
	Since the sequence $\{x_n\}$ is coarsely dense at infinity, there exist constants
	$D>0$ and $T_0>0$ such that for every $t\ge T_0$ there is an index $n$ with
	$|x_n - t|\le D$.
    Fix $m\in\mathbb{Z}$ and define
	$t_n := g(\lambda^{-m} y_n).$ Since $y_n\to\infty$ and $g$ is a quasi-isometry, we have $t_n\to\infty$.
	By coarse density, for all sufficiently large $n$ we may choose an index $r_n$
	such that
	$|x_{r_n} - t_n| \le D.$
	Applying the quasi-Lipschitz property of $g^{-1}$ yields
	\[
	|g^{-1}(x_{r_n}) - g^{-1}(t_n)|
	\le K|x_{r_n}-t_n| + A
	\le KD + A.
	\]
	On the other hand, since $t_n = g(\lambda^{-m}y_n)$ and $g^{-1}$ is a quasi-inverse of $g$,
	we have
	\[
	|g^{-1}(t_n) - \lambda^{-m}y_n| \le C.
	\]
	Combining the above estimates, there exists a constant $M>0$ such that
	\[
	\left| g^{-1}(x_{r_n}) - \frac{y_n}{\lambda^m} \right| \le M
	\quad \text{for all sufficiently large } n.
	\]
	Therefore,
	\[
	\left| \frac{\lambda^m g^{-1}(x_{r_n})}{y_n} - 1 \right|
	= \frac{\lambda^m}{y_n}
	\left| g^{-1}(x_{r_n}) - \frac{y_n}{\lambda^m} \right|
	\le \frac{\lambda^m M}{y_n}.
	\]
	Since $y_n\to\infty$, the right-hand side tends to $0$, and hence
	\[
	\lim_{n\to\infty}\frac{\lambda^m\, g^{-1}(x_{r_n})}{y_n}=1.\]
\end{proof}

The next lemma shows that asymptotic ratios are stable under asymptotic
equivalence of sequences.

\begin{Lem}
	\label{lem:stability}
	Let $g:\mathbb{R}_+\to\mathbb{R}_+$ be a quasi-isometry.
	Let $\{a_n\}$ and $\{b_n\}$ be sequences in $\mathbb{R}_+$ such that $a_n\to\infty$,
	$b_n\to\infty$, and
	$
	\displaystyle\lim_{n\to\infty}\displaystyle\frac{a_n}{b_n}=1.
	$
	Suppose further that the limit
	$
	\displaystyle\lim_{n\to\infty}\displaystyle\frac{g(a_n)}{a_n}=\ell
	$
	exists. Then
	$\displaystyle\lim_{n\to\infty}\frac{g(b_n)}{b_n}=\ell.$
\end{Lem}

\begin{proof}
	Since $g$ is a quasi-isometry, there exist constants $K\ge1$ and $C>0$ such that
	for all sufficiently large $u,v\in\mathbb{R}_+$,
	$|g(u)-g(v)| \le K|u-v| + C$.
	Fix $n$ sufficiently large. Then we obtain,
	

	\[
	\left|
	\frac{g(b_n)}{b_n} - \ell
	\right|
	\le
	\left|
	\frac{g(b_n)}{b_n} - \frac{g(a_n)}{a_n}
	\right|
	+
	\left|
	\frac{g(a_n)}{a_n} - \ell
	\right|.
	\tag{1}
	\]
	By assumption, the second term in (1) tends to zero as $n\to\infty$.
	Thus it suffices to show that
	\[
	\left|
	\frac{g(b_n)}{b_n} - \frac{g(a_n)}{a_n}
	\right|
	\longrightarrow 0.
	\]
	We write
	\[
	\begin{aligned}
		\left|
		\frac{g(b_n)}{b_n} - \frac{g(a_n)}{a_n}
		\right|
		&=
		\left|
		\frac{a_n g(b_n) - b_n g(a_n)}{a_n b_n}
		\right| \\
		&\le
		\frac{|a_n(g(b_n)-g(a_n))|}{a_n b_n}
		+
		\frac{|g(a_n)(a_n-b_n)|}{a_n b_n}.
	\end{aligned}
	\tag{2}
	\]
	For the first term in (2), using the quasi-Lipschitz property of $g$,
	we obtain
	\[
	\frac{|a_n(g(b_n)-g(a_n))|}{a_n b_n}
	\le
	\frac{K|b_n-a_n| + C}{b_n}.
	\]
	Since $\displaystyle\frac{a_n}{b_n}\to1$, we have,
	\[
	\frac{K|b_n-a_n| + C}{b_n} \longrightarrow 0.
	\tag{3}
	\]
	For the second term in (2), we write
	\[
	\frac{|g(a_n)(a_n-b_n)|}{a_n b_n}
	=
	\frac{g(a_n)}{a_n}\cdot\frac{|a_n-b_n|}{b_n}.
	\]
	By assumption, $\displaystyle\frac{g(a_n)}{a_n}\to\ell$, hence it is bounded,
	and again $\displaystyle\frac{|a_n-b_n|}{b_n}\to0$. Therefore,
	\[
	\frac{|g(a_n)(a_n-b_n)|}{a_n b_n} \longrightarrow 0.
	\tag{4}
	\]
	Combining (2), (3), and (4), we conclude that
	\[
	\left|
	\frac{g(b_n)}{b_n} - \frac{g(a_n)}{a_n}
	\right|
	\longrightarrow 0.
	\]
	Together with (1), this yields
	$
	\displaystyle\lim_{n\to\infty}\displaystyle\frac{g(b_n)}{b_n}=\ell.
	$
	
\end{proof}


Combining scale matching with stability of asymptotic ratios, we now establish
the main rigidity result for centralizers of dilations.

\begin{Thm}\label{main2}
	Let $f:\mathbb{R}_+\to\mathbb{R}_+$ be the dilation $f(x)=\lambda x$ with $\lambda>1$,
	and let $g\in QI(\mathbb{R}_+)$ satisfy $[g][f]=[f][g]$ in $QI(\mathbb{R}_+)$.
	Then the invariant $I{[g]}$ contains at most one element.
\end{Thm}

\begin{proof}
	Assume that there exists a coarsely dense sequence $\{x_n\}\subset\mathbb{R}_+$
	such that
	\[
	\lim_{n\to\infty}\frac{x_n}{g^{-1}(x_n)}=\ell.
	\]
	Let $\{y_n\}\subset\mathbb{R}_+$ be any sequence for which the limit
	$$\lim_{n\to\infty}\frac{y_n}{g^{-1}(y_n)}$$
	exists. We show that this limit must also be equal to $\ell$.	By Lemma~\ref{lem:scale-matching}, applied to the coarsely dense sequence
	$\{x_n\}$ and the sequence $\{y_n\}$, there exists a subsequence $\{x_{r_n}\}$
	such that
	\[
	\lim_{n\to\infty}\frac{\lambda^m g^{-1}(x_{r_n})}{y_n}=1.
	\]
	In particular, $\lambda^m g^{-1}(x_{r_n})$ and $y_n$ are asymptotically equivalent.\\
	Since $[g][f]=[f][g]$, for each $m\in\mathbb{Z}$ there exists a constant $C_m>0$
	such that
	\[
	|f^m(g(t)) - g(f^m(t))| \le C_m
	\quad \text{for all sufficiently large } t.
	\]
	Fix such an integer $m$. Note that $f^m(x)=\lambda^m x$ for all $x\in\mathbb{R}_+$. Then
	\[
	|g(\lambda^m s) - \lambda^m g(s)| \le C_m
	\quad \text{for all sufficiently large } s.
	\]
	Applying this with $s=g^{-1}(x_{r_n})$, we obtain
	\[
	|g(\lambda^m g^{-1}(x_{r_n})) - \lambda^m g(g^{-1}(x_{r_n}))| \le C_m.
	\]
	Since $g^{-1}$ is a quasi-inverse of $g$, there exists $A>0$ such that
	\[
	|g(g^{-1}(x_{r_n})) - x_{r_n}| \le A.
	\]
	Combining the above estimates yields
	\[
	|g(\lambda^m g^{-1}(x_{r_n})) - \lambda^m x_{r_n}|
	\le C_m + \lambda^m A.
	\]
	Dividing by $\lambda^m g^{-1}(x_{r_n})$ and using $g^{-1}(x_{r_n})\to\infty$,
	we conclude that

	\[
	\left|
	\frac{g(\lambda^m g^{-1}(x_{r_n}))}{\lambda^m g^{-1}(x_{r_n})}
	-
	\frac{\lambda^m x_{r_n}}{\lambda^m g^{-1}(x_{r_n})}
	\right|
	\le
	\frac{ C_m + \lambda^m A}{\lambda^m g^{-1}(x_{r_n})}.
	\]
	Since $g^{-1}(x_{r_n})\to\infty$, the right-hand side tends to $0$, and therefore
	\[
	\lim_{n\to\infty}
	\frac{g(\lambda^m g^{-1}(x_{r_n}))}{\lambda^m g^{-1}(x_{r_n})}
	=
	\lim_{n\to\infty}\frac{x_{r_n}}{g^{-1}(x_{r_n})}
	=\ell.
	\]
	Using the asymptotic equivalence
	$\lambda^m g^{-1}(x_{r_n}) \sim y_n$ and Lemma~\ref{lem:stability}, we conclude that
	\[
	\lim_{n\to\infty}\frac{g(y_n)}{y_n}=\ell.
	\]
	Finally, since $|g(g^{-1}(y_n)) - y_n| \le A$ for all sufficiently large $n$,
	we obtain
	\[
	\left|
	\frac{y_n}{g^{-1}(y_n)}
	-
	\frac{g(g^{-1}(y_n))}{g^{-1}(y_n)}
	\right|
	\le
	\frac{A}{g^{-1}(y_n)}
	\longrightarrow 0.
	\]
	Hence
	\[
	\lim_{n\to\infty}\frac{y_n}{g^{-1}(y_n)}=\ell.
	\]
	Therefore, any convergent limit defining an element of $I_{[g]}$ must equal
	$\ell$, and consequently $I{[g]}$ contains at most one element.
\end{proof}

\begin{Rem}
	The hypothesis of Theorem~\ref{main2} is not necessary for uniqueness of the
	asymptotic ratio. Consider the maps $f(x)=2x$ and $g(x)=2x+\log x$ on
	$\mathbb{R}_+$. It can be easily checked that both $f$ and $g$ are in $QI(\mathbb{R}_{+})$. Then
	\[
	\lim_{x\to\infty}\frac{f(x)}{x}=2
	\quad\text{and}\quad
	\lim_{x\to\infty}\frac{g(x)}{x}
	=
	2+\lim_{x\to\infty}\frac{\log x}{x}
	=2.
	\]
	Therefore, for every sequence $\{x_n\}$ with $x_n\to\infty$, we have
	\[
	\left|
	\frac{x_n}{g^{-1}(x_n)} -
	\frac{g(g^{-1}(x_n))}{g^{-1}(x_n)}
	\right|
	\le
	\frac{C}{g^{-1}(x_n)},
	\]
	where $C>0$ is a constant depending only on $g$. Since $g^{-1}(x_n)\to\infty$,
	the right-hand side tends to $0$, and therefore
	\[
	\lim_{n\to\infty}\frac{x_n}{g^{-1}(x_n)}=2.
	\]
	However,
	\[
	|f\circ g(x)-g\circ f(x)|
	=
	4x+2\log x-4x-\log(2x)
	=
	\log\!\left(\frac{x}{2}\right)
	\longrightarrow\infty
	\quad\text{as }x\to\infty,
	\]
	and therefore $[g]\notin C([f])$. This example shows that uniqueness of the
	asymptotic invariant alone does not imply coarse commutation with a dilation.
\end{Rem}
The previous remark illustrates that the centralizer condition is not necessary for uniqueness of the asymptotic ratio. We now show that it is not sufficient for existence of the invariant either. Indeed there exists $[g]\in C([f_\lambda])$ for which $I[g]=\varnothing$.
\begin{Ex}\label{Ig}
	Let $\lambda>1$ and let
	\[
	g(x)=x\left(1+\varepsilon\sin\left(\frac{2\pi\log x}{\log\lambda}\right)\right),
	\quad x>0, \text{ where } 0<\varepsilon<
	\frac{1}{1+\frac{2\pi}{\log\lambda}}.
	\]
	Then $g$ is a bi-Lipschitz self-homeomorphism of $\mathbb{R}_+$ satisfying
	$
	g(\lambda x)=\lambda g(x)
	$
	for all $x>0$. Consequently, $[g]\in C([f_\lambda])$, where
	$f_\lambda(x)=\lambda x$. Moreover,
	$I[g]=\varnothing$.
	
A direct computation gives
\[
g'(x)
=
1+\varepsilon\sin\!\left(\frac{2\pi\log x}{\log\lambda}\right)
+\frac{2\pi\varepsilon}{\log\lambda}
\cos\!\left(\frac{2\pi\log x}{\log\lambda}\right).
\]
Therefore,
\[
1-\varepsilon-\frac{2\pi\varepsilon}{\log\lambda}
\le g'(x)\le
1+\varepsilon+\frac{2\pi\varepsilon}{\log\lambda},
\]
and the choice of $\varepsilon$ implies that $g'(x)$ is bounded above and below by positive constants. Therefore $g$ is bi-Lipschitz.
Furthermore,
\[
g(\lambda x)
=
\lambda x
\left(
1+\varepsilon
\sin\left(
\frac{2\pi\log(\lambda x)}{\log\lambda}
\right)
\right)
=
\lambda x
\left(
1+\varepsilon
\sin\left(
2\pi+\frac{2\pi\log x}{\log\lambda}
\right)
\right)
=
\lambda g(x),
\]
so $g$ commutes with $f_\lambda$, and hence $[g]\in C([f_\lambda])$.

To show that $I[g]=\varnothing$, let $\{x_n\}$ be an arbitrary coarsely dense sequence at infinity. Then there exists $D>0$ such that for every sufficiently large $t>0$ there is an index $n$ with $|x_n-t|\le D$.
For each $k\ge1$, set
\[
a_k=\lambda^k,
\qquad
b_k=\lambda^{k+\frac14}.
\]
By coarse density, there exist subsequences $\{x_{n_k}\}$ and $\{x_{m_k}\}$ satisfying
\[
|x_{n_k}-a_k|\le D,
\qquad
|x_{m_k}-g(b_k)|\le D.
\]
Since $g(a_k)=a_k$ and $g(b_k)=(1+\varepsilon)b_k$, the bi-Lipschitz property of $g$ yields constants $M,N>0$ such that
\[
|g^{-1}(x_{n_k})-a_k|\le M,
\qquad
|g^{-1}(x_{m_k})-b_k|\le N.
\]
Consequently, by using above relations, 
\[
\frac{x_{n_k}}{g^{-1}(x_{n_k})}\longrightarrow 1,~\text{as}~k\to\infty,
\text{ while, }
\frac{x_{m_k}}{g^{-1}(x_{m_k})}\longrightarrow 1+\varepsilon,~\text{as}~k\to\infty.
\]
Since $1\neq 1+\varepsilon$, the sequence $\left\{\displaystyle\frac{x_n}{g^{-1}(x_n)}\right\}$ cannot converge. As the coarsely dense sequence $\{x_n\}$ was arbitrary, no coarsely dense sequence yields a convergent asymptotic ratio. Therefore
$I[g]=\varnothing$.
\end{Ex}
\begin{Rem}
	The coarse density condition is essential in the definition of $I[g]$. If one allows arbitrary sequences tending to infinity, then one may consider
	\[
	S_g=
	\left\{
	\lim_{n\to\infty}\frac{x_n}{g^{-1}(x_n)}
	:
	x_n\to\infty,\,
	\frac{x_n}{g^{-1}(x_n)}
	\text{ converges}
	\right\}.
	\]
	Observe that
	\[
	\frac{x_n}{g^{-1}(x_n)}
	\sim
	\frac{g(g^{-1}(x_n))}
	{g^{-1}(x_n)}.
	\]
	Hence, after writing $y_n=g^{-1}(x_n)$, we obtain
	\[
	S_g=
	\left\{
	\lim_{n\to\infty}\frac{g(y_n)}{y_n}
	:
	y_n\to\infty,\,
	\frac{g(y_n)}{y_n}
	\text{ converges}
	\right\},
	\]
	which coincides with the asymptotic scaling invariant introduced in \cite{Bhowmik Chakraborty 2}.
	If we consider the same $g$ defined as in Example \ref{Ig} then $I[g]=\varnothing$ where as $S_g$ is uncountable.
	For instance, $S_g\subseteq [1-\varepsilon,1+\varepsilon]$. Also
	let $t\in[-1,1]$. Choose $\theta\in[-\pi/2,\pi/2]$ such that
	$\sin\theta=t$, and define
	$x_n=\lambda^{\,n+\frac{\theta}{2\pi}}$. One can check that every point of the interval
	$[1-\varepsilon,1+\varepsilon]$
	belongs to $S_g$. Thus $S_g=[1-\epsilon,1+\epsilon]$.
\end{Rem}

	
	
	
	
	

\section{A Structure Theorem for a Subgroup of the Centralizer of a Dilation}

In the previous section, we introduced the asymptotic invariant $I[g]$ and established a rigidity theorem for quasi-isometries that coarsely commute with a dilation. In particular, Theorem 3.5 shows that if $[g]\in C([f_\lambda])$, then $I[g]$ contains at most one element. Thus, whenever the invariant exists, it determines a unique asymptotic scaling factor associated to $[g]$.
Since $I[g]$ need not be non-empty for every element $[g]$ of $C([f_\lambda])$, we restrict our attention to the subset
\[C^{*}([f_\lambda])=\{[g]\in C([f_\lambda]):I([g])\neq\varnothing\}.\]The purpose of this section is to show that the asymptotic invariant completely determines the algebraic structure of $C^{*}([f_\lambda])$. The following results shows that $C^{*}([f_\lambda])$ is indeed a subgroup of
$C([f_\lambda]).$

\begin{Lem}
	\label{thm:rigidity-general}
	Let $[g]\in C^{*}([f_\lambda])$ and suppose
	$
	I[g]=\{a\}.
	$
	Then there exists a constant $M>0$ such that for all $x$,
	\[
	\left|g^{-1}(x)-\frac{x}{a}\right|\le M.
	\]
\end{Lem}

\begin{proof}
	Define
	$$u(x)=g^{-1}(x)-\frac{x}{a}.$$
	Since $[g]\in C([f_\lambda])$, there exists a constant $C>0$ such that
	$|g^{-1}(\lambda x)-\lambda g^{-1}(x)|\le C,~\text{for~ all}~x.$
	 Substituting
	\[
	g^{-1}(x)=\frac{x}{a}+u(x),
	\]
	we obtain
	$|u(\lambda x)-\lambda u(x)|\le C$.
		Replacing $x$ by $\lambda^k x$, where $k\ge 0$ and 	summing over $k=0,1,\ldots,m-1$ and applying the triangle inequality, 
	\[
	\left|
	\frac{u(\lambda^{m}x)}{\lambda^{m}}
	-u(x)
	\right|
	\le
	\frac{C}{\lambda-1}
	\]
	for all  $x$ and every $m\in\mathbb N$.
	Since $I[g]=\{a\}$, there exists a coarsely dense sequence $\{x_n\}$ such that
	\[
	\lim_{n\to\infty}\frac{x_n}{g^{-1}(x_n)}=a.
	\]
	Using
	\[
	g^{-1}(x_n)=\frac{x_n}{a}+u(x_n), \qquad \text{ we get }
	\frac{u(x_n)}{x_n}\longrightarrow 0.
	\]
	Then from above,
	$|u(x)|\le \frac{C}{\lambda-1}$
	for all  $x$ (for detailed see proof of Lemma \ref{thm:injective} ). This implies $u$ is bounded, and therefore
	\[
	\left|g^{-1}(x)-\frac{x}{a}\right|
	\le M
	\]
	for some constant $M>0$ and for all $x$.
\end{proof}
\begin{Cor}
	\label{cor:subgroup}
	The set $C^{*}([f_\lambda])$ is a subgroup of $C([f_\lambda])$.
\end{Cor}

\begin{proof}
	Let $[g],[h]\in C^{*}([f_\lambda])$ and suppose
	$I[g]=\{a\}$ and
	$I[h]=\{b\}$.
	By Lemma~\ref{thm:rigidity-general} for all  $x$
	\[
	\left|g^{-1}(x)-\frac{x}{a}\right|\le M,
	\qquad
	\left|h^{-1}(x)-\frac{x}{b}\right|\le N
	\]
	 It follows that for all  $x$,
	\[
	\left|(gh)^{-1}(x)-\frac{x}{ab}\right|
	\le
	N+\frac{M}{b} \text{ hence, }
	\lim_{x\to\infty}
	\frac{x}{(gh)^{-1}(x)}
	=
	ab,
	\]
	and therefore $[gh]\in C^{*}([f_\lambda])$.
Similarly,
	\[
	\lim_{x\to\infty}\frac{x}{g(x)}
	=
	\frac1a,
\text{ that implies }
	I[g^{-1}] 
	=
	\left\{\frac1a\right\}.
	\]
	Thus $[g]^{-1}\in C^{*}([f_\lambda])$.
\noindent	
Since $I[\mathrm{id}_{\mathbb R_+}]=\{1\}$, the identity belongs to $C^{*}([f_\lambda])$. Therefore $C^{*}([f_\lambda])$ is a subgroup of $C([f_\lambda])$.
\end{proof}
By Theorem \ref{main2}, for every $[g]\in C^{*}([f_\lambda])$, there exists a unique $a>0$ such that
$
I[g]=\{a\}.
$
This allows us to define a map
$$
\Phi:C^{*}([f_\lambda])\longrightarrow \mathbb{R}_{+}
$$
 given by
$$
\Phi([g])=a,
\qquad\text{where } I[g]=\{a\}.
$$
Our goal is to prove that $\Phi$ is an isomorphism. We begin by showing that the invariant $I[g]$ behaves multiplicatively under composition.

\begin{Lem}
	\label{lemma:composition}
	Let $[g],[h]\in C^{*}([f_\lambda])$ and suppose
	$I[g]=\{a\}$ and
	$I[h]=\{b\}$.
	Then
	\[
	I[gh]=\{ab\}.
	\]
\end{Lem}

\begin{proof}
	Let $\{x_n\}$ be a coarsely dense sequence at infinity such that
	$$
	\lim_{n\to\infty}\frac{x_n}{(gh)^{-1}(x_n)}
	$$
	exists. Since
	$
	(gh)^{-1}=h^{-1}\circ g^{-1},
	$
	we have
	$$
	\frac{x_n}{(gh)^{-1}(x_n)}
	=\frac{x_n}{g^{-1}(x_n)}
	\cdot
	\frac{g^{-1}(x_n)}
	{h^{-1}(g^{-1}(x_n))}.
	$$
	Since $g^{-1}$ is a quasi-isometry and $\{x_n\}$ is coarsely dense at infinity, the sequence $\{g^{-1}(x_n)\}$ is also coarsely dense at infinity. Therefore,
	$$
	\lim_{n\to\infty}\frac{x_n}{g^{-1}(x_n)}=a
	\qquad\text{and}\qquad
	\lim_{n\to\infty}
	\frac{g^{-1}(x_n)}
	{h^{-1}(g^{-1}(x_n))}
	=b.
	$$
	Passing to the limit in the above identity yields
	$$
	\lim_{n\to\infty}
	\frac{x_n}{(gh)^{-1}(x_n)}=ab.
	$$
	Therefore,
	$
	I[gh]=\{ab\}.
	$
\end{proof}
The multiplicative property established in Lemma \ref{lemma:composition} allows us to identify the asymptotic invariant with a group homomorphism.
\begin{Prop}\label{homomorphism}
	The map
	$\Phi:C^{*}([f_\lambda])\to \mathbb{R}_{>0}$
	is a well-defined group homomorphism.
\end{Prop}

\begin{proof}
	By Theorem 3.5, for every $[g]\in C^{*}([f_\lambda])$, the set $I[g]$
	contains a unique element. Hence $\Phi$ is well defined.
	
	Let $[g],[h]\in C^{*}([f_\lambda])$ and suppose
	$
	I[g]=\{a\},
	$ and $
	I[h]=\{b\}.
	$
	By Lemma \ref{lemma:composition},
	$
	I[gh]=\{ab\}.
	$
	Therefore
	$$
	\Phi([gh])
	=
	ab
	=
	\Phi([g])\Phi([h]),
	$$
	and hence $\Phi$ is a group homomorphism.
\end{proof}

We next prove that every positive real number arises as the asymptotic invariant of an element of $C^{*}([f_\lambda])$.

\begin{Prop}
\label{prop:surjective}
The map
$
\Phi:C^{*}([f_\lambda])\to \mathbb R_{>0}
$
is surjective.
\end{Prop}
\begin{proof}
	Let $a>0$ and consider the dilation
	$
	g_a(x)=ax.
	$
	Since $g_a$ is bi-Lipschitz, we have $[g_a]\in QI(\mathbb R_+)$. Moreover,
	$
	f_\lambda\circ g_a=g_a\circ f_\lambda,
	$
	and hence $[g_a]\in C([f_\lambda])$.
	Since
	$$
	g_a^{-1}(x)=\frac{x}{a},
	\text{ we get, } 
	\frac{x_n}{g_a^{-1}(x_n)}=a
	$$
	for every coarsely dense sequence $\{x_n\}$. Therefore,
	$
	I[g_a]={a}.
	$
	Consequently, $[g_a]\in C^{*}([f_\lambda])$ and
	$
	\Phi([g_a])=a.
	$
	Since $a>0$ was arbitrary, $\Phi$ is surjective.
\end{proof}

We now turn to the kernel of $\Phi$. The following auxiliary lemma provide the estimates needed to establish injectivity which can be proved similarly as Lemma \ref{thm:rigidity-general}.

\begin{Lem}
	\label{lem:kernel1}
	Let $[h]\in C^{*}([f_\lambda])$ and define
	$$
	u(x)=h^{-1}(x)-x.
	$$
	Then there exists a constant $C>0$ such that
	$$
	|u(\lambda x)-\lambda u(x)|\le C,~
	\text{for~ all~} x.$$ Hence for all  $x$ and every $n\in\mathbb N$,
	$$
	\left|
	\frac{u(\lambda^n x)}{\lambda^n}
	-u(x)
	\right|
	\le
	\frac{C}{\lambda-1}.
	$$
	
\end{Lem}

The preceding lemma provide the estimates needed to determine the kernel of $\Phi$.
\begin{Lem}
	\label{thm:injective}
	The homomorphism
	$\Phi:C^{*}([f_\lambda])\to \mathbb R_{>0}$	
	is injective.
\end{Lem}

\begin{proof}
	Let $[h]\in\ker(\Phi)$. Then
	$
	\Phi([h])=1,
	$
	and hence
	$
	I[h]=\{1\}.
	$\\
	Define
	$$
	u(x)=h^{-1}(x)-x.
	$$
	Let $\{x_n\}$ be a coarsely dense sequence at infinity. Since
	$
	I[h]=\{1\},
	$
	we have
	$$
	\lim_{n\to\infty}\frac{x_n}{h^{-1}(x_n)}=1.
	$$
	Since $h^{-1}(x_n)=x_n+u(x_n)$, it follows that
	$
	\displaystyle\lim_{n\to\infty}\displaystyle\frac{u(x_n)}{x_n}=0.
	$\\
	Fix $t>0$. Since $\{x_n\}$ is coarsely dense at infinity, there exist
	$D>0$ and a subsequence $\{x_{n_m}\}$ such that
	$$
	|x_{n_m}-\lambda^m t|\le D
	$$
	for all sufficiently large $m$.\\
	Since $h^{-1}$ is a quasi-isometry, there exist constants $K\ge1$ and
	$A>0$ such that
	$$
	|u(a)-u(b)|
	\le
	(K+1)|a-b|+A
	$$
	for all sufficiently large $a,b$.
	Consequently,
	$$
	|u(x_{n_m})-u(\lambda^m t)|
	\le
	(K+1)D+A.
	$$
	Dividing by $\lambda^m$ and letting $m\to\infty$, we obtain
	$$
	\frac{u(\lambda^m t)}{\lambda^m}\longrightarrow0,
	$$
	since
	$$
	\frac{u(x_{n_m})}{\lambda^m}
	=
	\frac{u(x_{n_m})}{x_{n_m}}
	\cdot
	\frac{x_{n_m}}{\lambda^m}
	\longrightarrow0.
	$$
	Applying Lemma \ref{lem:kernel1}, we obtain
	$$
	\left|
	\frac{u(\lambda^m t)}{\lambda^m}-u(t)
	\right|
	\le
	\frac{C}{\lambda-1}.
	$$
	Letting $m\to\infty$ yields
	$$
	|u(t)|
	\le
	\frac{C}{\lambda-1}.
	$$
	Since $t>0$ was arbitrary, $u$ is bounded on $\mathbb R_+$. Therefore
	$
	h^{-1}\sim id_{\mathbb R_+},
	$
	and hence $[h]=1$. Thus $\ker(\Phi)=\{1\}$, proving that $\Phi$ is injective.
\end{proof}
The preceding results completely determine the structure of
$C^{*}([f_\lambda])$.

	\subsection{Proof of Theorem \ref{main}} By Proposition \ref{homomorphism}, the map
	$$
	\Phi:C^{*}([f_\lambda])\to\mathbb R_{+}
	$$
	is a group homomorphism. Proposition \ref{prop:surjective}
	shows that $\Phi$ is surjective, while
	Lemma \ref{thm:injective} implies that $\Phi$ is injective.
	Hence $\Phi$ is an isomorphism, and therefore
	$$
	C^{*}([f_\lambda])
	\cong
	\mathbb R_{+}.
	$$

\begin{Rem}
	This classification illustrates the necessity of restricting attention to
	$C^{*}([f_\lambda])$ rather than the full centralizer $C([f_\lambda])$.
	As demonstrated in Example~3.7, the centralizer $C([f_\lambda])$
	contains bi-Lipschitz homeomorphisms whose asymptotic invariants are empty,
	that is, $I[g]=\varnothing$. Consequently, the map $\Phi$ cannot be defined
	on all of $C([f_\lambda])$. Passing to $C^{*}([f_\lambda])$ isolates those
	elements exhibiting asymptotically stable scaling behavior, and Theorem~4.4
	shows that this distinguished subgroup is naturally isomorphic to
	$\mathbb{R}_{>0}$. 
\end{Rem}
\begin{Rem}
	As an immediate consequence of Theorem~4.4, the subgroup
	$C^{*}([f_\lambda])$ is left-orderable. Indeed, since
	$
	C^{*}([f_\lambda])\cong \mathbb{R}_{>0},
$
	it inherits the natural order structure of $\mathbb{R}_{>0}$. In fact,
	$C^{*}([f_\lambda])$ is bi-orderable.
\end{Rem}

\begin{Rem}
	Theorem~\ref{main} identifies the subgroup $C^{*}([f_\lambda])$
	algebraically with the multiplicative group $\mathbb{R}_{+}$ via the
	isomorphism
	\[
	\Phi:C^{*}([f_\lambda])\longrightarrow\mathbb{R}_{+}.
	\]
	Since $\mathbb{R}_{+}$ carries its standard left order and
	$C^{*}([f_\lambda])$ is a subgroup of the left-orderable group
	$QI(\mathbb{R}_{+})$ \cite{Ye Zhao}, it is natural to ask whether this
	algebraic identification is also order-theoretic. More precisely, one may ask
	whether the restriction of a left order on $QI(\mathbb{R}_{+})$ to
	$C^{*}([f_\lambda])$ agrees with the order transported from
	$\mathbb{R}_{+}$ via $\Phi$. Equivalently, is the isomorphism
	$\Phi$ order-preserving with respect to the induced left order on
	$C^{*}([f_\lambda])$? A positive answer would establish a direct link
	between the asymptotic rigidity developed in the present paper and the
	order-theoretic structure of the ambient quasi-isometry group.
\end{Rem}

\section*{Acknowledgements}
The authors thank Dr. Prateep Chakraborty for his valuable suggestions and fruitful comments. The first author of this article acknowledges the financial support from Indian Institute of Science Education and Research Bhopal, India and second author acknowledges the financial support from IIT Palakkad, India.

\section*{Conflict of Interest}
The author declares that there is no conflict of interest regarding the publication of this article.

\end{document}